\documentclass[11pt]{amsart}
\usepackage{amssymb,amsmath}
\usepackage{enumitem,graphicx}
\usepackage{accents}

\usepackage{amssymb}
\usepackage{amsmath,amscd,amssymb,latexsym}
\usepackage{indentfirst}
\usepackage{pstricks}
\usepackage{setspace}
\usepackage{latexsym}
\usepackage{amsthm}
\usepackage{amsmath}
\usepackage{color,enumitem,graphicx,amsthm}
\usepackage[colorlinks=true,urlcolor=blue,
citecolor=red,linkcolor=blue,linktocpage,pdfpagelabels,
bookmarksnumbered,bookmarksopen]{hyperref}
\usepackage[english]{babel}
\newcommand{\ld} {\langle}
\newcommand{\rd} {\rangle}
\newcommand{\mc} {\mathcal}

\renewcommand{\epsilon}{\varepsilon}

\newcommand{\la} {\lambda}

\newcommand{\noi} {\noindent}

\newcommand{\mb}{\mathbb}

\newcommand{\ds} {\displaystyle}

\newtheorem{defi}{Definition}[section]
\newtheorem{thm}{Theorem}[section]
\newtheorem{thmlet}{Theorem}

\newtheorem{lem}{Lemma}[section]
\newtheorem{pro}{Proposition}[section]
\newtheorem{cor}{Corollary}[section]
\usepackage[all]{xy}
\catcode`\@=11
\def\theequation{\@arabic{\c@section}.\@arabic{\c@equation}}
\catcode`\@=12

\def\R{{I\!\!R}}

\date{}

\title[Fractional Kirchhoff problem]{Fractional Kirchhoff problem with critical indefinite nonlinearity}
\author[ P. K. Mishra]{Pawan Kumar Mishra}
\author[J. M. do \'O]{Jo\~ao Marcos do \'O }
\author[X. He]{Xiaoming He}
\address[P. K. Mishra]{ Department of Mathematics, Federal University of Para\'{i}ba\newline\indent 58051-900, Jo\~ao Pessoa-PB, Brazil }
\email{\href{mailto:pawanmishra31284@gmail.com}{pawanmishra31284@gmail.com}}
\address[J. M. do \'O]{ Department of Mathematics, Federal University of Para\'{i}ba\newline\indent 58051-900, Jo\~ao Pessoa-PB, Brazil }
\email{\href{mailto:jmbo@pq.cnpq.br}{jmbo@pq.cnpq.br}}
\address[X. He]{ Department of Mathematics,College of Science, Minzu University of China\newline\indent 100081, Beijing, P.R. of China }
\email{\href{mailto:xmhe923@muc.edu.cn}{xmhe923@muc.edu.cn}}

\thanks{Research supported in part by INCTmat/MCT/Brazil, CNPq and CAPES/Brazil.}

\begin{document}
\begin{abstract}
We study the existence and multiplicity of positive solutions for a family of fractional Kirchhoff equations with critical nonlinearity of the form
\begin{equation*}
M\left(\int_\Omega|(-\Delta)^{\frac{\alpha}{2}}u|^2dx\right)(-\Delta)^{\alpha} u= \lambda f(x)|u|^{q-2}u+|u|^{2^*_\alpha-2}u\;\; \text{in}\; \Omega,\;u=0\;\textrm{in}\;\mathbb R^n\setminus \Omega,
\end{equation*}
where $\Omega\subset \mathbb R^n$ is a smooth bounded domain, $ M(t)=a+\varepsilon t, \; a, \; \varepsilon>0,\;
0<\alpha<1,  \; 2\alpha<n<4\alpha$ and $ \; 1<q<2$. Here $2^*_\alpha={2n}/{(n-2\alpha)}$ is the fractional critical Sobolev exponent,  $\lambda$ is a positive parameter and the coefficient $f(x)$ is a real valued continuous function which  is allowed to change sign. By using a variational approach based on the idea of Nehari manifold technique, we combine effects of a sublinear and a superlinear term to prove our main results.
\end{abstract}
\maketitle
\section{Introduction}
\noindent
This paper is concerned with the existence and multiplicity of positive solutions for the
 family of fractional Kirchhoff equations
\begin{equation}\label{Q}
  M\left(\int_\Omega|(-\Delta)^{\frac{\alpha}{2}}u|^2dx\right)(-\Delta)^{\alpha} u = \lambda f(x)|u|^{q-2}u+|u|^{2^*_\alpha-2}u\;\; \text{in}\; \Omega,\;u=0\;\textrm{in}\;\mathbb R^n\setminus \Omega,
\end{equation}
 where $ \Omega\subset \mathbb R^n$ is a smooth bounded domain,
  $M(t)=a+\varepsilon t, \; a,\;  \varepsilon>0, \;  0<\alpha<1,\; 2\alpha<n<4\alpha,\;  1<q<2$ and
  $2^*_\alpha={2n}/{(n-2\alpha)} $ is the fractional critical  Sobolev exponent. Here the coefficient $f(x)$ is a real valued continuous function which is allowed to change sign,  $\lambda$ is a positive parameter and $(-\Delta)^{\alpha}$ is the fractional Laplacian operator defined in section 2.
   A basic feature of the family of problems considered here is that it has double nonlocal structure
   due to the presence of the fractional Laplacian and the nonlocal Kirchhoff function $M$ which makes the equation no longer
   a pointwise identity. Moreover, the nonlinear term
$g_{\lambda}(x,s)=\lambda f(x)|s|^{q-2}s+|s|^{2^*_\alpha-2}s $ is sublinear� at $0$ and superlinear at $\infty$ with critical growth.

\noi A lot of attention has been given to the study of elliptic equations involving
fractional Laplace operator because of pure mathematical research and its wide range of applications in many branches of Science. Non-local operators naturally arise in continuum mechanics, phase transition phenomena, population dynamics and game theory, see \cite{MR3289358} and references therein. Fractional operators are also involved in financial mathematics, where Levy processes with jumps appear in modeling the asset prices, see \cite{MR2512800}. The fractional Laplacian is the infinitesimal generator of L\'evy stochastic processes.  This is also of interest in Fourier analysis where it is defined as a pseudo-differential operator. Moreover, these operators arise in a quite natural way in many different physical situations in which one has to consider long range anomalous diffusions and transport in highly heterogeneous medium.
In the last decade many authors studied the existence and multiplicity of solutions for nonlocal problems involving fractional powers of  Laplacian $(-\Delta)^\alpha,  \alpha\in (0,1) $. We cite \cite{MR2646117, MR2354493, MR3379042, MR3271254, MC, MR3165278, MR2911424, MR3345604} with no attempts to provide the complete list of references. There are several works related to nonlocal problems of Kirchhoff type  involving fractional Laplacian, see \cite{MR3120682, MR3373607, MR3470662, MR3456737} and references therein. In \cite{MR3120682}, authors have given the motivation for the fractional Kirchhoff type operators by studying the string vibrations. Moreover, using the concentration-compactness principle  authors have proved the existence result for the  critical exponent problem with superlinear perturbation.\\
In the case \textbf{$M\equiv 1$} and \textbf{$\alpha=1$}, T. F. Wu \cite{MR2373222} has considered the following local problem
\begin{align}\label{8w}
-\Delta u=\lambda f(x)|u|^{q-2}u+|u|^{2^*-2}u, \;\;u\in H^1_0(\Omega),
\end{align}
where $\Omega\subset \mathbb R^n$, $n\geq3, 2>q>1$, $f(x)$ is continuous sign changing weight, $\lambda>0$ is a parameter and $2^*={2n}/{(n-2)}$ is the critical Sobolev exponent. Using the Nehari manifold technique, the following result has been obtained.
\begin{thmlet}
There exists a $\lambda_0>0$ such that problem \eqref{8w} has at least two positive solutions for $\lambda \in (0, \lambda_0)$.
\end{thmlet}
\noi In the case $M\equiv 1$ and $\alpha \in (0, 1)$  similar results for fractional Laplacian has been studied in \cite{MC}. In \cite{MC}, authors have considered the following problem
\begin{align}\label{rw}
(-\Delta)^\alpha u=\lambda f(x)|u|^{q-2}u+|u|^{2_\alpha^*-2}u\;\textrm{in}\;\Omega, \;\;u=0\;\;\;\;\;\;\textrm{on}\;\mathbb R^n\setminus \Omega,
\end{align}
where $\Omega\subset \mathbb R^n$, $n\geq2>q>1$, $f(x)$ is continuous sign changing weight, $\lambda>0$ is a parameter and $2^*_\alpha $ is the fractional critical Sobolev exponent. Using the harmonic extension technique authors have extended the multiplicity results obtained in \cite{MR2373222} to the nonlocal problem \eqref{rw}.\\
\noi In the case $M\not \equiv 1$ and $\alpha=1$, there is a lot of work
 addressed by many researchers, see \cite{MR2763559, MR3490847, MR3210026} and references therein. Recently in \cite{MR3490847}, authors have shown the multiplicity result for Kirchhoff type problems, without assuming any sign changing weight, with the restriction on the coefficient of the Kirchhoff term. Precisely authors have considered the following problem in $\mathbb R^3$ with $f(x)\equiv 1$
 \begin{align*}
-\left(a+\varepsilon \displaystyle \int_\Omega |\nabla u|^2dx\right)\Delta u(x)=u^5+\lambda u^{q-1}, \;u>0\;\textrm{in}\;\Omega, \;\;u=0\;\;\;\;\;\;\textrm{on}\; \partial\Omega,
\end{align*}
 where $\Omega\subset \mathbb R^3$ is smooth bounded domain, $a>0, \varepsilon>0$ is sufficiently small and $\lambda>0$ is a positive parameter. The following multiplicity result was proved
 \begin{thmlet}
 Assume $a > 0, 1<q <2$ and $\varepsilon> 0$ is sufficiently small. Then there exists $\lambda_* > 0$ such that
for any $\lambda\in (0, \lambda_*)$, problem \eqref{rw} has at least two positive solutions, and one of the solutions is a positive ground state solution.
 \end{thmlet}
\noi In this paper we have studied the multiplicity results for the case $ M\not \equiv 1$ and $\alpha \in (0, 1)$ with a restriction on the coefficient  of Kirchhoff term for sign changing weight. We face some  difficulties in solving this type of problems. First one is the presence of Kirchhoff term in the energy functional which makes study of Palais-Smale sequence rather complicated because the weak limit of minimizing sequence is no more a weak solution of the problem as in the case for $M\equiv 1$. We need strong convergence in order to show that the limit of minimizing sequence is the minimizer of the energy functional in case of Kirchhoff problems. The second one is the lack of compactness because of the critical Sobolev growth. The Sobolev embedding is continuous but not compact for critical exponent, $2^*_\alpha$. Apart from this we have nonlocal nature of the problem because of presence of nonlocal fractional operator as well. Motivated by the work of Caffarelli and Silvestre \cite{MR2354493}, we have considered an equivalent definition of the fractional operator in a bounded domain with zero Dirichlet boundary
data by means of an auxiliary variable but in the process we lack of explicit form of extremal functions. We have used the estimates as in \cite{MR2911424, MR3117361} on the extension of the extremal functions while using concentration compactness Lemma due to Lions \cite{MR834360} to get the compactness of Palais-Smale sequence. We have adopted the idea of Nehari manifold to obtain the existence of two solutions for suitable choice of positive parameters $\lambda$ and sufficiently small $\varepsilon>0$. For the details related to Nehari manifold and fibering map analysis, see \cite {MR1998965, MR2373222, MR2557956}. To the best of our knowledge, the multiplicity results  for fractional Kirchhoff type problems with critical exponent and sign changing weight, obtained in this paper, has not been established before. \\
 \noi With this introduction, we state the main result of the paper in the form of following theorem.
\begin{thm}\label{22mht1}
Assume $a>0$ and $\varepsilon>0$ sufficient small. Then we have the following
 \begin{enumerate}[label=(\roman*)]
\item  There exists a $\lambda_0>0$ such that Problem \eqref{Q} has at least one positive solution with negative energy for $\lambda \in (0, \lambda_0)$.
\item{There exists a $\lambda_{00}>0$ such that for $0<\lambda<\lambda_{00}\leq\lambda_0$,  Problem \eqref{Q} has at least two positive solutions.}
\end{enumerate}
\end{thm}
\noindent The paper is organized as follows. In section 2, we have discussed the variational formulation of the problem and the functional setting. In section 3, we have discussed Nehari manifold and fibering map analysis and proved the existence of first solution. In section 4, we have shown the existence of second solution and concluded the proof of our Theorem \ref{22mht1}.

\section{Variational formulation and functional setting}
\noi The fractional powers of Laplacian, $(-\Delta)^{\alpha}$, in a bounded domain $\Omega$ with zero Dirichlet boundary data are defined through the spectral decomposition using the powers of the eigenvalues of the
Laplacian operator. Let $(\varphi_j,\rho_j)$ be the eigen pair of $(-\Delta)$ in $\Omega$ with zero Dirichlet boundary data. Then
$(\varphi_j,\rho_j^{\alpha})$ is the eigen pair of
$(-\Delta)^{\alpha}$ with  Dirichlet boundary conditions. In fact, the
fractional Laplacian $(-\Delta)^{\alpha}$ is well defined in the space of
functions
$$
H^{\alpha}_0(\Omega)=\left\{u=\sum a_j
  \varphi_j\in L^2(\Omega)\  :\  \|u\|_{H^{\alpha}_0(\Omega)}=
  \left(\sum a_j^2\rho_j^{\alpha}\right)^{1/2}<\infty\right\}
$$
and, as a consequence,
$$
(-\Delta)^{\alpha}u=\sum
  a_j\rho_j^{\alpha}\varphi_j\,.
$$
Note that then
$\|u\|_{H^{\alpha}_0(\Omega)}=\|(-\Delta)^{\alpha/2}u\|_{L^2(\Omega)}$. The dual space $H^{-\alpha}(\Omega)$ is defined in the standard way, as well as the inverse operator $(-\Delta)^{-\alpha}$.
The variational functional associated to the problem $\eqref{Q}$ is given as
\begin{equation*}\label{fpl}
\mc J_{M,\lambda}(u)=\frac{1}{2}\widehat M\left(\int_\Omega|(-\Delta)^{\frac{\alpha}{2}}u|^2dx\right)-\frac{\lambda}{q}\int_\Omega f(x)|u|^q dx-\frac{1}{2^*_\alpha}\int_\Omega |u|^{2^*_\alpha} dx,
\end{equation*}
where $\widehat M(t)=\displaystyle \int_0^t M(s)ds$ is the primitive of $M$.
\begin{defi}
A function $u\in H^\alpha_{0}(\Omega)$ is called a weak solution of the problem \eqref{Q} if for all $\phi \in H^\alpha_0(\Omega)$ the following holds
\begin{equation*}
M\left(\int_\Omega|(-\Delta)^{\frac{\alpha}{2}}u|^2dx\right)\displaystyle \int_\Omega (-\Delta)^\frac{\alpha}{2} u (-\Delta)^\frac{\alpha}{2}\phi dx=\lambda \int_\Omega f(x)|u|^{q-2}u\phi dx+\int_\Omega |u|^{2^*_\alpha-2} u \phi dx.
\end{equation*}
\end{defi}

\noindent Recently a powerful technique is developed by Caffarelli and Silvestre \cite{MR2354493} to treat the nonlocal problems involving fractional Laplacian. In this technique, we study an extension problem corresponding to a nonlocal problem so that we can investigate the nonlocal problem via classical variational methods. In this work we use this harmonic extension technique. We first define the harmonic extension of $u\in H^\alpha_0(\Omega)$.
 \begin{defi}
 For $u\in H^\alpha_0(\Omega)$, the harmonic extension $E_{\alpha}(u):=w$ is the solution of the following problem
 \begin{equation*}
\left\{\begin{array}{rlll}
-\mathrm{div}(y^{1-2\alpha}\nabla w) &=0 \; \;\; \text{in}\;\;\mathcal C\;\;=\Omega\times (0, \infty),\\
w&=0\;\;\; \text{on}\;\; \partial_L=\partial \Omega\times (0, \infty),\\
w&=u\;\;\; \text{on}\; \;\;\;\;\;\;\;\;\;\;\Omega \times\{0\},
\end{array}
\right.
\end{equation*}
where $\partial_L$ denotes the lateral boundary of the cylinder $\Omega \times (0, \infty)$. Moreover, the extension function is related with fractional Laplacian by
\[
(-\Delta)^\alpha u(z)=-\kappa_\alpha\displaystyle \lim_{y\rightarrow 0^+}y^{1-2\alpha}\frac{\partial w}{\partial y}(z,y),
\]
where $\kappa_\alpha$ is a normalization constant.
 \end{defi}
 \noi The solution space for the extension problem is
$$
H_{0, L}^1(\mathcal{C})=\left\{w\in
L^{2}(\mathcal{C})\,:\: w=0 \mbox{ on }
\partial_L,\;  \|w\|_{H_{0, L}^1(\mathcal{C})}:=\|w\|=
\left(\kappa_\alpha\int_{\mathcal{C}} y^{1-2\alpha}|\nabla
w|^2\right)^{1/2}<\infty\right\}.
$$
We observe that the extension operator is an isometry
between $H_0^{\alpha}(\Omega)$ and
$H_{0, L}^1(\mathcal{C})$. That is
\begin{equation}\label{equival norm}
\|E_\alpha(u)\|=\|u\|_{H^{\alpha}_{0}(\Omega)}\,
, \quad \forall \, u\in H_{0}^{\alpha}(\Omega).
\end{equation}
This isometry in \eqref{equival norm} is the key to study the Kirchhoff type problems in the harmonic extension set up. Moreover, we have the following trace inequality.\\

\noi \textbf{Trace  inequality:} For any function $\psi\in
H_{0, L}^1(\mathcal{C}),$ it holds $\|\psi(\cdot,0)\|_{H^{\alpha}_{0}(\Omega)}\leq\|\psi\|.$

\noindent In the subsequent Lemmas we use the following trace inequality.
\begin{lem}\label{22traceemb}
Let $2\leq r\leq 2^*_\alpha $, then there exists $C_r>0$ such that for all $v\in H^1_{0, L}(\mathcal C)$,
\begin{equation*}
\left(\displaystyle \int_{\mathcal C}y^{1-2\alpha}|\nabla v|^2 dzdy\right)^\frac12\geq C_r\left(\int_{\Omega\times \{0\}}|v(z, 0)|^r dx\right)^{\frac{1}{r}}.
\end{equation*}
Moreover, for $r=2^*_\alpha$, the best constant in Lemma \ref{22traceemb} will be denoted by $S(\alpha, n)$ and  it is indeed achieved in the case $\Omega=\mathbb{R}^{N+1}_+$ when
$u=\mathrm{trace}\;v=v(\cdot,0)$ takes the form
\begin{equation}\label{extrml}
u(x)=u_{\varepsilon}(x)= \frac{\varepsilon^{(N-2\alpha)\slash
2}}{(|x|^2+\varepsilon^2)^{(N-2\alpha)\slash
    2}}
\end{equation}
with $\varepsilon>0$ arbitrary.

\end{lem}

\noi As discussed above, the problem \eqref{Q} is equivalent to the study of the following extension problem
  \begin{equation}\label{E}
\left\{\begin{array}{rl}
 -\mathrm{div} (y^{1-2\alpha}\nabla w)&=0, \quad\textrm{in}\; \mathcal{C},\\
 w &=0 \;\;\;\; \text{on}\; \partial_L,\\
  M(\|w\|^2)\frac{\partial w}{\partial \nu}&= \lambda f(z)|w|^{q-2}w+|w|^{2^*_\alpha -2}w\;\; \textrm{on}\;\; \Omega\times \{0\},
\end{array}
\right.
\end{equation}
where $\frac{\partial w}{\partial \nu}=-\kappa_\alpha\displaystyle\lim_{y\rightarrow 0^+}y^{1-2\alpha}\frac{\partial w}{\partial y}(z, y)$.

\noindent The functional $\mathcal I_{M, \lambda}:H^1_{0, L}(\mathcal C)\rightarrow \mathbb R$ associated to the problem $\eqref{E}$ is defined as
\begin{equation}\label{fel}
\mathcal I_{M, \lambda}(w)=\frac{1}{2}\widehat M(\|w\|^2) -\frac{\lambda}{q}\int_{\Omega\times \{0\}} f(z)|w(z, 0)|^q dz-\frac{1}{2^*_\alpha}\int_{\Omega\times\{0\}} |w(z,0)|^{2^*_\alpha}dz.
\end{equation}
Any function $w\in H^1_{0, L}(\mathcal C)$ is called the weak solution of the problem $\eqref{E}$ if for all $\phi \in H^1_{0, L}(\mathcal C)$
\begin{align*}
M(\|w\|^2)\kappa_\alpha \int_\mathcal C y^{1-2\alpha}\nabla w.\nabla \phi dzdy&=\lambda \int_{\Omega\times\{0\}} f(z)|w(z, 0)|^{q-2}w(z, 0)\phi(z, 0) dz\\&+\int_{\Omega\times \{0\}}|w(z,0)|^{2^*_\alpha-2}w(z, 0)\phi(z, 0) dz.
\end{align*}
It is clear that critical points of $\mathcal I_{M, \lambda}$ in $H^1_{0, L}(\mathcal C)$ corresponds to the critical points of $\mc J_{M,\lambda}$ in $H^\alpha_0(\Omega)$. Thus if $w$ solves the problem \eqref{E}, then $u=\mathrm{trace}(w)=w(z,0)$ is the solution of the problem \eqref{Q} and vice-versa.
Therefore we look for the solutions  $w$ of extended problem $\eqref{E}$ to get the solutions of the problem $\eqref{Q}$. \\

\section{Nehari manifold and fibering maps}
\noi Now we consider the Nehari manifold associated to the problem $\eqref{E}$ as
\begin{equation*}
\mc N_\lambda=\{w\in H^1_{0, L}(\mathcal C)\setminus\{0\}\; : \; \langle \mc I_{M,\lambda}^{\prime}(w),w\rangle=0\}.
\end{equation*}
Thus $w\in \mc N_\lambda$ if and only if
\begin{equation}\label{22nlam}
M(\|w\|^2)\|w\|^2-\lambda \int_{\Omega\times \{0\}}f(z)|w(z, 0)|^q dz-\int_{\Omega\times \{0\}}|w(z,0)|^{2^*_\alpha}dz=0.
\end{equation}
Now  for a fixed $w\in H^1_{0, L}(\mathcal C)$ we define the fiber map $\Phi_w: \mathbb{R}^+\rightarrow \mathbb{R}$ as $\Phi_w(t)=\mc I_{M,\lambda}(tw)$. Thus $tw\in \mc N_\lambda$ if and only if
\begin{equation*}
\Phi_w^{\prime}(t)=tM(t^2\|w\|^2)\|w\|^2 -\lambda t^{q-1}\int_{\Omega\times \{0\}}f(z)|w(z, 0)|^q dz-t^{2^*_\alpha-1}\int_{\Omega\times \{0\}}|w(z,0)|^{2^*_\alpha} dz=0.
\end{equation*}

Also
\begin{equation}\label{22phiseconde}
\Phi^{\prime\prime}_w(1)=a\|w\|^2+3\varepsilon\|w\|^4-(q-1)\lambda\int_{\Omega\times \{0\}}f(z)|w(z, 0)|^q dz-{(2^*_\alpha-1)}\int_{\Omega\times \{0\}}|w(z, 0)|^{2^*_\alpha} dz.
\end{equation}
We split $\mc N_\lambda$ into three parts as
\begin{equation*}\label{nlaz}
\mc N_\lambda^{+}=\{w\in \mc N_\lambda\;|\;\Phi^{\prime\prime}_w(1)> 0\}, \;\;\mc N_\lambda^{-}=\{w\in \mc N_\lambda\;|\;\Phi^{\prime\prime}_w(1)< 0\} \; \text{and}\; \mc N_\lambda^0=\{w\in \mc N_\lambda\;|\;\Phi^{\prime\prime}_w(1)=0\}.\\
\end{equation*}
In general the set $\mc N_\lambda$ is not a manifold but following Lemma shows that it is indeed a $ C^1-$ manifold.
\begin{lem}\label{l22.2}
There exists $\lambda_{1} > 0$ such that $\mathcal{N}_{\lambda}^{0} = \emptyset, \;\textrm{for all} \;\lambda \in (0, \lambda_{1})$.
\end{lem}
\begin{proof}
We have following two cases.\\
\noi \textbf{Case 1:}  $w \in \mathcal{N}_{\lambda}$ and $ \ds \int_{\Omega\times \{0\}}f(z) |w(z,0)|^{q}dz = 0.$\\
  From \eqref{22nlam} , we have,\; $a \|w\|^{2} + \varepsilon\|w\|^{4}- \ds \int_{\Omega\times\{0\}} |w(z,0)|^{{2^*_\alpha}} dz = 0$. Now,
  \begin{eqnarray*}
    a \|w\|^{2} + 3\varepsilon\|w\|^{4}- (2^*_\alpha-1)\int_{\Omega\times\{0\}}|w|^{{2^*_\alpha}} dz  &=& (2-{2^*_\alpha})a \|w\|^{2} + (4-{2^*_\alpha})\varepsilon\|w\|^{4}<0
  \end{eqnarray*}
which implies $w \notin \mathcal{N}_{\lambda}^{0}$.\\
  \noi \textbf{Case 2:} $w \in \mathcal{N}_{\lambda}$ and $\ds \int_{\Omega\times\{0\}}f(z) |w(z,0)|^{q}dz \neq 0.$\\
  Suppose $w \in \mathcal{N}_{\lambda}^{0}$ . Then from \eqref{22nlam} and \eqref{22phiseconde}, we have
    \begin{eqnarray}\label{2.11}
    (2-q)a \|w\|^{2} + (4-q)\epsilon\|w\|^{4} &=& ({2^*_\alpha}-q)\int_{\Omega\times\{0\}}|w(z,0)|^{2^*_\alpha} dz,\\\label{2.12}
    ({2^*_\alpha}-2) a \|w\|^{2}+({2^*_\alpha}-4)\varepsilon\|w\|^4 &=& ({2^*_\alpha}-q) \lambda \int_{\Omega\times\{0\}} f(z)|w(z,0)|^q dz.
  \end{eqnarray}
  Define $E_{\lambda}: \mc N_{\lambda}\rightarrow \mathbb R$ as
  \begin{equation*}
    E_{\lambda}(w) = \frac{({2^*_\alpha}-2) a \|w\|^2+({2^*_\alpha}-4)\varepsilon \|w\|^4}{{2^*_\alpha}-q} - \lambda\int_{\Omega\times\{0\}} f(z)|w(z,0)|^{q} dz,
  \end{equation*}
 { then from \eqref{2.12}, $E_{\lambda}(w) = 0, \; \textrm{for all}\; w \in \mathcal{N}_{\lambda}^{0}.$  Also,}
  \begin{eqnarray*}
  E_{\lambda}(w)& \geq & \frac{({2^*_\alpha}-2)a\|w\|^2}{{2^*_\alpha}-q} - \lambda {\|f\|_{{\frac{{{2^*_\alpha}}}{{{2^*_\alpha}}-q}}}}\|w\|^q(\sqrt{\kappa_\alpha} S(\alpha, n))^{-q},\\
  & \geq & \|w\|^q\left(\left(\frac{{2^*_\alpha}-2}{{2^*_\alpha}-q}\right)a\|w\|^{(2-q)} - \lambda {\|f\|_{{\frac{{2^*_\alpha}}{{2^*_\alpha}-q}}}}(\sqrt{\kappa_\alpha} S(\alpha, n))^{-q}\right).\\
  \end{eqnarray*}

  Now, from \eqref{2.11}, we get
  \begin{equation}\label{2.13}
  \|w\| \geq \left(\left(\frac{2-q}{{2^*_\alpha}-q}\right){a (\sqrt{\kappa_\alpha} S(\alpha, n))^{{2^*_\alpha}}}\right)^{\frac{1}{{2^*_\alpha}-2}}.
  \end{equation}
  From \eqref{2.13}, there exists $\lambda_1>0$ such that for $\lambda\in (0, \lambda_1),\;E_{\lambda}(w)>0, \; \forall\; w \in \mathcal{N}_{\lambda}^{0},$ which is contradiction.
  \end{proof}

  \noi The following lemma shows that minimizers for $\mc I_{M, \lambda}$ on any
subset of $\mc N_{\la}$ are usually critical points for $\mc I_{M, \lambda}$.
%%%%%%%%%%%%%%%%%%%%%%%%%%%%%%%%%%%%%%%%%%%%%%%%%%%%%%%%%%%%%%%%%%%%%%%%%%%%%%%%%%%%%%%%%%%%%%%%
\begin{lem}
Let $w$ be a local minimizer for $\mc I_{M, \lambda}$ in any of the subsets of $\mc N_{\la}$ such that
$w\notin \mc N_{\la}^{0}$, then $w$ is a critical point
for $\mc I_{M, \lambda}$.
\end{lem}
\begin{proof}
Let $w$ be a local minimizer for $\mc I_{M, \lambda}$ in any of the
subsets of $\mc N_{\la}$. Then, in any case $w$ is a
minimizer for $\mc I_{M, \lambda}$ under the constraint $\mc J_{M, \la}(w):=\ld
\mc I_{M, \lambda}^{\prime}(w),w\rd =0$. Hence, by the theory of Lagrange
multipliers, there exists $\mu \in \mb R$ such that $ \mc I_{M, \lambda}^{\prime}(w)= \mu \mc J_{M, \la}^{\prime}(w)$. Thus $\ld
\mc I_{M, \lambda}^{\prime}(w),w\rd= \mu\;\ld \mc J_{M, \lambda}^{\prime}(w),w\rd = \mu
\Phi_{w}^{\prime\prime}(1)$=0, but $w\notin \mc N_{\la}^{0}$ and consequently
$\Phi_{w}^{\prime\prime}(1) \ne 0$. Hence $\mu=0$ which completes the
proof of the Lemma.
\end{proof}

\begin{lem}\label{cbb}
$\mathcal I_{M,\lambda}$ is coercive and bounded below on  $\mathcal{N}_{\lambda}$. Moreover, there exists a
constant $ C >0$ such that $\mathcal I_{M,\lambda} > - C \lambda^{2/(2-q)}.$
\end{lem}
\begin{proof}
For $w\in \mathcal{N}_{\lambda},$ we have
\begin{align*}
\mathcal I_{M,\lambda}(w) &= \left(\frac12-\frac{1}{2^*_\alpha}\right)a\|w\|^{2}+\left(\frac14-\frac{1}{{2^*_\alpha}}\right)\varepsilon \|w\|^4 - \lambda \left(\frac{1}{q}-\frac{1}{{2^*_\alpha}}\right)\int_{\Omega\times\{0\}} f(z)|w(z,0)|^{q}dz,\\
&\geq \left(\frac12-\frac{1}{{2^*_\alpha}}\right)a\|w\|^{2} - \lambda \left(\frac{1}{q}-\frac{1}{{2^*_\alpha}}\right)(\sqrt{\kappa_\alpha}S(\alpha, n))^{-q}\|f\|_{\frac{{2^*_\alpha}}{{2^*_\alpha}-q}}\|w\|^q.
\end{align*}
Define
$$
g(t) =\left(\frac12-\frac{1}{{2^*_\alpha}}\right)at^{\frac{2}{q}} - \lambda \left(\frac{1}{q}-\frac{1}{2^*_\alpha}\right)(\sqrt{\kappa_\alpha} S(\alpha, n))^{-q}\|f\|_{\frac{2^*_\alpha}{2^*_\alpha-q}}t,
$$
then $g(t)$ attains its minimum at
$$t=\left(\frac{\lambda(2^*_\alpha-q)\|f\|_{\frac{2^*_\alpha}{2^*_\alpha-q}}(\sqrt{\kappa_\alpha} S(\alpha, n))^{-q}}{(2^*_\alpha-2)a}\right)^\frac{q}{2-q}.
$$
Hence $\mc I_{M,\lambda}(w)\geq-C\lambda^\frac{2}{2-q}$
for some constant $C>0$.
\end{proof}
\noi Define
$$
H^+=\left\{w\in H^1_{0,L}(\mathcal C): \ds \int_{\Omega\times\{0\}} f(z)|w(z,0)|^qdz>0\right\}
$$
and
$$
H^-=\left\{w\in H^1_{0,L}(\mathcal C): \ds \int_{\Omega\times\{0\}} f(z)|w(z,0)|^qdz<0\right\}.
$$
Then we have the following lemma
\begin{lem}\label{L37}
(i) For every $w \in H^+$, there exists $\lambda_2>0$, unique $t_{\max}=t_{\max}(w)>0$
and unique $t^+(w)<t_{\max}<t^-(w)$ such that $t^+w\in \mathcal{N}_\lambda^+, t^-w\in \mathcal{N}_\lambda$  for $\lambda \in (0, \lambda_2)$ and
$\mathcal I_{M,\lambda}(t^+w)=\displaystyle\min_{0\leq t\leq t^-} \mathcal I_{M,\lambda}(tw)$,
$\mathcal I_{M,\lambda}(t^-w) = \displaystyle \max_{t\geq t_{\max}} \mathcal I_{M,\lambda}(tw)$.\\
(ii) For $w \in H^-$, there exists a unique $t^*>0$ such that $t^*w\in \mathcal{N}_\lambda^-$.
\end{lem}
\begin{proof}
Define $\psi_w:\mathbb{R}^+\rightarrow \mathbb{R}$ as
\begin{align}\nonumber
&\psi_w(t)=at^{2-q}\|w\|^2+\epsilon t^{4-q}\|w\|^{4}-t^{2^*_\alpha-q}\int_{\Omega\times\{0\}} |w(z,0)|^{2^*_\alpha} dz.\;\textrm{Then}\;\\
&\psi_w^{\prime}(t)=a(2-q)t^{1-q}\|w\|^2+\epsilon (4-q)t^{3-q}\|w\|^{4}-(2^*_\alpha-q)t^{2^*_\alpha-1-q}\int_{\Omega\times\{0\}} |w(z,0)|^{2^*_\alpha} dz.\label{siutd}
\end{align}
We also note that $\Phi_{tw}$ and $\psi_w$ satisfies $\Phi_{tw}^{\prime\prime}(1)=t^{-q-1}\psi_w^{\prime}(t).$
Let $w \in H^{+}.$ Then from \eqref{siutd}, we note that $\psi_{w}(t) \rightarrow - \infty$ as $t \rightarrow \infty$. From \eqref{siutd}, it is easy to see that
$\displaystyle\lim_{t\rightarrow 0^{+}}\psi'_{w}(t)>0$ and $\displaystyle\lim_{t\rightarrow \infty}\psi^{'}_{w}(t)<0$. Moreover, it can be shown that there exists
a unique $t_{\max} = t_{\max}(w)>0$ such that $\psi_{w}(t)$ is increasing on $(0, t_{\max})$, decreasing on $(t_{\max}, \infty)$ and
$\psi'_{w}(t_{\max}) = 0$, that is,
\begin{equation}\label{2secnew}
    a(2-q)t_{\max}^{2}\|w\|^{2} + \epsilon (4-q)t_{\max}^{4}\|w\|^{4}-({2^*_\alpha}-q)t^{2^*_\alpha}_{\max}\int_{\Omega\times\{0\}}{|w(z,0)|^{2^*_\alpha}}dz, = 0.
\end{equation}
which implies
\begin{equation}\label{eneq}
t_{\max} \geq \frac{1}{\|w\|} \left(\frac{a(2-q)(\sqrt{\kappa_\alpha} S(\alpha, n))^{2^*_\alpha}}{({2^*_\alpha}-q)}\right)^{\frac{1}{2^*_\alpha-2}}:=T_1.
\end{equation}
Using inequality {\eqref{eneq}}, we have
\begin{eqnarray*}
    \psi_{w}(t_{\max}) & \geq & \psi_w(T_1)\geq a T_1^{2-q}\|w\|^2-T_1^{{2^*_\alpha}-q}\int_{\Omega\times\{0\}} |w(z,0)|^{{2^*_\alpha}}dz, \\
    & \ge  & C\|w\|^q \left(\left(\frac{2-q}{2^*_\alpha-q}\right)^\frac{2-q}{2^*_\alpha-2}-\left(\frac{2-q}{2^*_\alpha-q}\right)^\frac{2^*_\alpha-q}{2^*_\alpha-2}\right)>0.
\end{eqnarray*}
Hence there exists a $\lambda_2>0$ such that if
$\lambda<\lambda_2$, there exists unique $t^+ = t^+(w) {<} t_{\max}$ and $t^- = t^-(w) > t_{\max},$ such that
$\psi_{w}(t^+) = \lambda \ds \int_{\Omega\times\{0\}}{f(z)|w(z,0)|^{q}}dz = \psi_{w}(t^-)$.
That is,  $t^+w, t^-w \in \mathcal{N}_{\lambda}.$ Also $\psi'_{w}(t^+) > 0$ and $\psi_{w}'(t^-) < 0$
implies $t^+w \in \mathcal{N}^{+}_{\lambda}$ and $t^-w \in \mathcal{N}^{-}_{\lambda}.$
 Since
 $$
 \Phi'_{w}(t) = t^{q}\left(\psi_{w}(t)- \lambda \ds \int_{\Omega\times\{0\}} f(z)|w(z,0)|^{q}dz\right).
 $$
  Then $\Phi'_{w}(t)<0$ for all $t \in [0, t^+)$
  and $\Phi'_{w}(t)>0$ for all $t \in (t^+, t^-)$. So $\mathcal I_{M,\lambda}(t^+w) = \displaystyle\min_{0 \leq t \leq t^-}\mathcal I_{M,\lambda} (tw).$
   Also $\Phi'_w(t) > 0$ for all $t \in [t^+, t^-),
\Phi'_w(t^-) = 0$ and $\Phi'_w(t) < 0$ for all $t \in (t^-, \infty)$ implies that $\mathcal I_{M,\lambda}(t^-w)
= \displaystyle\max_{t \geq t_{\max}} \mathcal I_{M,\lambda}(tw).$\\
(ii) Let $w\in H^-$. Then from \eqref{siutd},
 we note that $\psi_{w}(t) \rightarrow -\infty$ as $t \rightarrow \infty$.
  Hence for all $\lambda>0$ there exists $t^*>0$ such that $t^*w\in \mathcal{N}_\lambda^-$.
\end{proof}
\noi Define
$$\theta _\lambda=\inf\{\mc I_{M,\lambda}(w): w\in \mathcal N_\lambda\},\;\;\theta _\lambda^+=\inf\{\mc I_{M,\lambda}(w): w\in \mathcal N_\lambda^+\} \;\;\textrm{and}\;\; \theta _\lambda^-=\inf\{\mc I_{M,\lambda}(w): w\in \mathcal N_\lambda^-\}.
$$
Then we have the following Lemma.
\begin{lem}\label{3a}
There exists $\mathrm{C} > 0$ such that $\theta_{\lambda}^+ < -\left(\frac12-\frac{1}{2^*_\alpha}\right)\frac{(2-q)}{q} a \mathrm{C}.$
\end{lem}
\begin{proof}
Let $v_{\lambda} \in H^1_{0,L}(\mathcal C)$ such that $\ds \int_{\Omega\times \{0\}}{f(z)|v_{\lambda}(z,0)|^{q}dz}>0$. Then by Lemma \ref{L37},
there exists unique $ t_{\lambda}(v_{\lambda}) > 0$ such that  $t_{\lambda}v_{\lambda} \in \mathcal{N}_\lambda^{+}$. Now from \eqref{22nlam} and \eqref{22phiseconde}, we have
\begin{equation*}
 \mathcal I_{M,\lambda}(t_{\lambda}v_{\lambda}) = \left(\frac{1}{2}-\frac{1}{q}\right)a \|t_{\lambda} v_{\lambda}\|^{2} + \left(\frac{1}{4}-\frac{1}{q}\right)\epsilon \|t_{\lambda} v_{\lambda}\|^{4} + \left(\frac{1}{q}-\frac{1}{2^*_\alpha}\right) \int_{\Omega\times\{0\}} {|t_{\lambda}v_{\lambda}(z,0)|^{2^*_\alpha}dz}.
\end{equation*}
and
\begin{equation*}
  \int_{\Omega\times\{0\}} {|t_{\lambda}v_{\lambda}(z,0)|^{2^*_\alpha}dz} \leq \left(\frac{2-q}{2^*_\alpha-q}\right)a\|t_\lambda v_\lambda\|^{2}+ \left(\frac{4-q}{2^*_\alpha-q}\right)\varepsilon\|t_\lambda v_\lambda\|^{4}.
\end{equation*}
Therefore
\begin{align*}
\mathcal I_{M,\lambda}(t_{\lambda}v_{\lambda}) &\leq -\left(\frac12-\frac{1}{2^*_\alpha}\right)\frac{(2-q)}{q} a\|t_{\lambda}v_{\lambda}\|^{2}-\left(\frac14-\frac{1}{2^*_\alpha}\right)\frac{(4-q)}{q} \varepsilon \|t_{\lambda}v_{\lambda}\|^{4},\\
&\leq-\left(\frac12-\frac{1}{2^*_\alpha}\right)\frac{(2-q)}{q} a \mathrm{C},
\end{align*}
where $\mathrm{C}=\|t_\lambda v_\lambda\|^2.$ This implies $ \theta_{\lambda}^+ \leq -\left(\frac12-\frac{1}{2^*_\alpha}\right)\frac{(2-q)}{q} a \mathrm{C}$.
\end{proof}
\noi Concerning the component set $\mc N_\lambda^-$, we have the following Lemma which helps us to show that the set $\mc N_\lambda^-$ is closed in the $H^1_{0, L}(\mc C)$ topology.
\begin{lem}\label{nlw}
There exists $\delta>0$ such that $\|w\|\geq \delta$ for all $w\in \mc N_\lambda^-$.
\end{lem}
\begin{proof}
Let $w\in N_\lambda^-$ then from \eqref{22phiseconde}, we get
\begin{align*}
a\|w\|^2+3\epsilon\|w\|^4-\lambda(q-1)\int_{\Omega \times\{0\}}f(x)|w(z,0)|^qdz<(2^*_\alpha-1)\int_{\Omega \times\{0\}}|w(z,0)|^{2^*_\alpha}dz.
\end{align*}
 Now using \eqref{22nlam}, and $2\sqrt {ab}\leq (a+b)$ together with the Lemma \ref{22traceemb}, we get
 \begin{align*}
 2\sqrt{a\epsilon(2-q)(4-q)}\|w\|^3&\leq (2-q)a\|w\|^2+(4-q)\epsilon\|w\|^4<(2^*_\alpha-q)\int_{\Omega \times\{0\}}|w(z,0)|^{2^*_\alpha}dz,\\
 &<(2^*_\alpha-q)(\sqrt{\kappa_\alpha} S(\alpha, n))^{-2^*_\alpha}\|w\|^{2^*_\alpha}
 \end{align*}
 which implies that $\|w\|^{2^*_\alpha-3}>C$. Hence $\|w\|\geq \delta$ for some $\delta>0$.
 \end{proof}
\begin{cor}\label{nlclosed}
$\mc N_\lambda^-$ is closed set in the $H^1_{0, L}(\mc C)$ topology.
\end{cor}
\begin{proof}
Let $\{w_k\}$ be a sequence in $\mc N_\lambda^-$ such that $w_k\rightarrow w$ in $H^1_{0, L}(\mc C)$. Then $w\in \overline {\mc N_\lambda^-}=\mc N_\lambda^-\cup\{0\}$. Now using Lemma \ref{nlw}, we get $\|w\|=\displaystyle \lim_{k\rightarrow \infty}\|w_k\|\geq \delta>0$. Hence $w\neq 0$. Therefore $w\in \mc N_\lambda^-$.
\end{proof}

\begin{lem}\label{zii}
For a given $w \in \mathcal{N}_{\lambda}$ and $\lambda \in (0, \lambda_{1}),$ there exists $\epsilon > 0$ and a differentiable function
$\xi : \mathcal{B}(0,\epsilon) \subseteq H^1_{0,L}(\mathcal C) \rightarrow \mathbb{R}$ such that $\xi(0)=1,$ the function $\xi(v)(u-v)\in \mathcal{N}_{\lambda}$
and
\begin{equation}\label{xiie}
\langle\xi^{'}(0), v\rangle = \frac{2a \langle w, v\rangle+ 4\epsilon \|w\|\langle w, v\rangle -  \ds \int_{\Omega\times\{0\}} \left(q\lambda f(z)|w(z,0)|^{q-2}+ 2^*_\alpha|w(z,0)|^{2^*_\alpha-2}\right)w(z,0)\;v(z,0) dz}{(2-q)a\|w\|^{2} + (4-q)\epsilon\|w\|^{4}-(2^*_\alpha-q)\ds \int_{\Omega\times\{0\}} |w(z,0)|^{2^*_\alpha}dz},
\end{equation}
where $$\langle w, v\rangle=\kappa_\alpha\ds \int_{\mathcal C} y^{1-2\alpha}\nabla w\nabla v dz\,dy\notag
\;\;\textrm{for all}\;\; v\in H^1_{0,L}(\mathcal C).$$
\end{lem}
\begin{proof}
For fixed $u\in \mathcal{N}_\lambda$, define $\mathcal{F}_u:\mathbb{R}\times H^1_{0,L}(\mathcal C)\rightarrow \mathbb{R}$ as follows
\begin{eqnarray*}
  \mathcal{F}_u(t,w) = t^{2}a\|u-w\|^{2}+t^{4}\epsilon\|u-w\|^{4}&-&t^{q}\lambda \int_{\Omega\times\{0\}}{f(z)|(u-w)(z,0)|^{q}dz}\\&-&t^{2^*_\alpha}\int_{\Omega\times\{0\}} {|(u-w)(z,0)|^{2^*_\alpha}dz},
\end{eqnarray*}
then $\mathcal{F}_u(1,0) = 0,\; \frac{\partial}{ \partial t}\mathcal{F}_u(1,0)\neq 0$ as $\mathcal{N}_{\lambda}^{0} = \emptyset$ for $\lambda\in (0, \lambda_1)$. So we can apply implicit function theorem to get a differentiable function $\xi : \mathcal{B}(0, \epsilon) \subseteq H^1_{0,L}(\mathcal C) \rightarrow \mathbb{R}$ such that $\xi(0) = 1$
and \eqref{xiie} holds and $\mathcal{F}_u(\xi(w),w) = 0$, $\textrm{for all}\; w \in \mathcal{B}(0, \epsilon)$. Hence $ \xi(w)(u-w) \in \mathcal{N}_{\lambda}$.
\end{proof}
\noi Now using the Lemma \ref{zii}, we prove the following proposition which shows the existence of Palais-Smale sequence.

\begin{pro}\label{prp1}
Let $\lambda \in (0,\lambda_{3}).$ Then there exists a minimizing sequence $\{w_k\} \subset \mathcal{N}_{\lambda}$ such that
\begin{center}
    $\mathcal I_{M,\lambda}(w_{k}) = \theta_{\lambda}+o_k(1)$ and $\mathcal I_{M,\lambda}^{'}(w_{k}) = o_k(1).$
\end{center}
\end{pro}
\begin{proof}

 From Lemma \ref{cbb}, $\mc I_{M,\lambda}$ is bounded below on $\mathcal N_\lambda$. So by Ekeland variational principle, there exists a minimizing sequence $\{w_k\}\in \mathcal N_\lambda$ such that
\begin{eqnarray}
\mathcal I_{M,\lambda}(w_k)&\leq& \theta_{\lambda}+\frac{1}{k},\label{Pek1}\\
\mathcal I_{M,\lambda}(v)&\geq& \mc I_{M,\lambda}(w_k)- \frac{1}{k}\|v-w_k\|\;\;\mbox{for all}\;\;v\in \mc N_{\lambda}.\nonumber
\end{eqnarray}
Using \eqref{Pek1} and Lemma \ref{3a}, it is easy to show that $w_k\not\equiv 0$. From Lemma \ref{cbb},  we have that $\displaystyle\sup_{k}\Vert w_k\Vert<\infty$. Next we claim that $\|\mathcal I_{M,\lambda}^\prime(w_k)\|\rightarrow 0$ as $k \rightarrow 0$.
Now, using the Lemma \ref{zii} we get the differentiable functions $\xi_k:\mathcal{B}(0, \epsilon_k)\rightarrow \mathbb{R}$ for some $\epsilon_k>0$ such that $\xi_k(v)(w_k-v)\in \mc {N}_\lambda$,\; $\textrm{for all}\;\; v\in \mathcal{B}(0, \epsilon_k).$
For fixed $k$, choose $0<\rho<\epsilon_k$. Let $w\in H^1_{0, L}(\mathcal C)$ with $w\not\equiv 0$ and let $v_\rho={\rho w}/{\|w\|}$. We set $\eta_\rho=\xi_k(v_\rho)(w_k-v_\rho)$. Since $\eta_\rho \in \mc {N}_\lambda$, we get from \eqref{22nlam}
\begin{align*}
\mathcal I_{M,\lambda}(\eta_\rho)-\mathcal I_{M,\lambda}(w_k)\geq-\frac{1}{k}\|\eta_\rho-w_k\|.
\end{align*}
Now by mean value theorem, we get
\begin{equation*}
\langle \mathcal I_{M,\lambda}^{\prime}(w_k),\eta_\rho-w_k\rangle+o_k(\|\eta_\rho-w_k\|)\geq -\frac{1}{k}\|\eta_\rho-w_k\|.
\end{equation*}
Hence
\begin{align*}
\langle \mathcal I_{M,\lambda}^{\prime}(w_k),-v_\rho\rangle+(\xi_k(v_\rho)-1)\langle \mathcal I_{M,\lambda}^{\prime}(w_k),(w_k-v_\rho)\rangle \geq -\frac{1}{k}\|\eta_\rho -w_k\|+o_k(\|\eta_\rho -w_k\|)
\end{align*}
and since $\langle \mathcal I_{M,\lambda}^{\prime}(\eta_\rho),(w_k-v_\rho)\rangle=0$, we have
\begin{align*}
-\rho\langle \mathcal I_{M,\lambda}^{\prime}(w_k),\frac{w}{\|w\|}\rangle&+(\xi_k(v_\rho)-1)\langle \mathcal I_{M,\lambda}^{\prime}(w_k)-\mathcal I_{M,\lambda}^{\prime}(\eta_\rho),(w_k-v_\rho)\rangle \\&\geq -\frac{1}{k}\|\eta_\rho-w_k\|+o_k(\|\eta_\rho-w_k\|).
\end{align*}
Thus
\begin{align}
\langle \mathcal I_{M,\lambda}^{\prime}(w_k),\frac{w}{\|w\|}\rangle \leq \frac{1}{k\rho}\|\eta_\rho-w_k\|&+\frac{o_k(\|\eta_\rho-w_k\|)}{\rho}\nonumber\\&+
\frac{(\xi_k(v_\rho)-1)}{\rho}\langle  \mathcal I_{M,\lambda}^{\prime}(w_k)-\mathcal I_{M,\lambda}^{\prime}(\eta_\rho),(w_k-v_\rho)\rangle\label{fifte}.
\end{align}
Since
$\displaystyle
\|\eta_\rho-w_k\|\leq \rho|\xi_k(v_\rho)|+|\xi_k(v_\rho)-1|\|w_k\|
$
and
$$
\displaystyle\lim_{\rho\rightarrow 0^+}\frac{|\xi_k(v_\rho)-1|}{\rho}\leq \|\xi_k'(0)\|,$$
taking limit  $\rho\rightarrow 0^+$\ in \eqref{fifte}, we get
\begin{equation*}
\langle \mathcal I_{M,\lambda}^\prime(w_k),\frac{w}{\|w\|}\rangle\leq\frac{C}{k}(1+\|\xi_k^{'}(0)\|)
\end{equation*}
for some constant $C>0$, independent of $w$. So if we can show that $\|\xi_k^{'}(0)\|$ is bounded then we are done.
 Now from Lemma \ref{zii}, (Note that from Lemma \ref{cbb} and Lemma \ref{3a}, $\|w_k\|\leq C\lambda$) the boundedness of $\{w_k\}$ and H\"older's inequality, for some $K>0$, we get for $0<\lambda<\lambda_3$ with $\lambda_3<\lambda_1$ small enough
\begin{align*}
\langle \xi^{\prime}(0),v\rangle= \frac{K\|v\|}{(2-q)a\|w_k\|^{2} + (4-q)\epsilon\|w_k\|^{4}-(2^*_\alpha-q)\ds \int_{\Omega\times\{0\}}|w_k(z,0)|^{2^*_\alpha}dz}.
\end{align*}
So to prove the claim we only need to prove that the denominator in the above expression
  is bounded away from zero. Suppose not. Then there exists a subsequence, still denoted by $\{w_k\}$, such that
\begin{equation}\label{baw}
(2-q)a\|w_k\|^{2} + (4-q)\epsilon\|w_k\|^{4}-(2^*_\alpha-q)\int_{\Omega\times\{0\}} |w_k(z,0)|^{2^*_\alpha}dz=o_k(1).
\end{equation}
From \eqref{baw} we get $E_\lambda(w_k)=o_k(1)$. Now using the fact that $\|w_k\|\geq C>0$ and following the proof of Lemma \ref{l22.2} we get $E_\lambda (w_k)>C_1$ for all $k$ for some $C_1>0$, which is a contradiction.
\end{proof}
\noi In order to prove compactness of Palais-Smale sequence we need the following result (see Theorem 5.1, \cite{MR2911424}).
\begin{thmlet}\label{lions}
Let $\{w_{k}\}_{k\in\mathbb{N}}$ be a weakly convergent sequence to $w_0$ in
$H^1_{0,L}(\mc C)$, such that the
sequence $\{y^{1-2\alpha}|\nabla w_{k}|^{2}\}_{k\in\mathbb{N}}$ is tight.  Let $u_k=\mathrm{trace}\;(w_{k})$
and $u_0=\mathrm{trace}\;(w_0)$. Let $\mu$, $\nu$ be two non negative measures such that
 \begin{equation*}
 y^{1-2\alpha}|\nabla w_{k}|^{2} \rightarrow \mu \qquad \mbox{and} \qquad |w_{k}|^{2^{*}_{\alpha}} \to \nu,\quad \mbox{as }\, k\rightarrow \infty
 \end{equation*}
in the sense of measures. Then there exist an at most countable set $J$ and points $\{x_{j}\}_{j\in J}\subset\Omega$, positive constants $\mu_{j}>0, \nu_{j}>0$ such that
\begin{align*}
 \displaystyle \nu = |w|^{2^{*}_{\alpha}} + \sum_{j\in J}
\nu_{j}\delta_{x_{j}},\;\;\;\displaystyle \mu
\geq y^{1-2\alpha}|\nabla w_0|^{2} + \sum_{j\in J}
\mu_{j}\delta_{x_{j}},\;\;\; \displaystyle\mu_{j}\geq S(\alpha,n)\nu_{j}^{\frac{2}{2^*_{\alpha}}}.
\end{align*}
\end{thmlet}
\noi In order to apply the concentration-compactness result, Theorem
\ref{lions}, first we prove the following.
%%%%%%%%%%%%%%%%%%%%%%%%%%%%%%%%%%%%%%
\begin{lem}\label{tight1}
The sequence $\left\{y^{1-2\alpha}|\nabla w_{k}|^{2}\right\}_{k\in\mathbb{N}}$ is tight, i.e., for any $\eta>0$ there exists $\rho_0>0$ such that
\begin{equation*}
\int_{\{y>\rho_0\}}{\int_{\Omega}{y^{1-2\alpha}|\nabla w_{k}|^{2}dxdy}}\le \eta,\quad \forall\,k\in\mathbb{N}.
\end{equation*}
\end{lem}
\begin{proof}
The proof of this lemma follows similar arguments  as of Lemma 3.6 in \cite{MR2911424}. By contradiction, we suppose that there exits $\eta_{0}>0$ such that, for any $\rho>0$ one has, up to a subsequence,
\begin{equation}\label{hhh}
\int_{\{y>\rho\}}{\int_{\Omega}{y^{1-2\alpha}|\nabla w_{k}|^{2}dxdy}}>\eta_{0}\,\quad\mbox{for every }\, k\in\mathbb{N}.
\end{equation}
Let $\varepsilon>0$ be fixed, and let $r>0$ be  such that
\[\int_{\{y>r\}}{\int_{\Omega}{y^{1-2\alpha}|\nabla w_k|^{2}dxdy}}<\varepsilon.\]
Let $j=\left[\frac{K}{\kappa_{\alpha}\varepsilon}\right]$ be the integer part
and $I_{i}=\{y\in\mathbb{R}^{+}:\, \,r+i\leq y\leq r+i+1\}$,
$i=0,\,1,\,\ldots,\,j$. Since
$\|w_{k}\|\leq K$, we clearly obtain
that
$$\sum_{i=0}^{j}{\int_{I_{i}}{\int_{\Omega}{y^{1-2\alpha}|\nabla w_{k}|^{2}dxdy}}}\leq \int_{\mathcal{C}_\Omega}y^{1-2\alpha}|\nabla w_{k}|^{2}dxdy\leq\varepsilon(j+1).$$
Therefore there exists $i_{0}\in\{0,\,\ldots,\,j\}$ such that, up to a
subsequence,
\begin{equation}\label{tight}
\int_{I_{i_{0}}}{\int_{\Omega}{y^{1-2\alpha}|\nabla w_{k}|^{2}dxdy}}\leq\varepsilon,\quad \forall\, k.
\end{equation}
Let $\chi\ge 0$ be the following regular non-decreasing smooth cut-off function such that $\chi(y)=0$ if $y\leq r+i_{0}$ and $\chi(y)=1$ if $y>r+i_{0}+1$. Define
$v_{k}(x,y)=\chi(y)w_{k}(x,y)$. Since $v_{k}(x,0)=0$  it follows that
\begin{align*}
&|\langle \mc I_{M,\lambda}'(w_{k})-\mc I_{M,\lambda}'(v_{k}),v_{k}\rangle|\\ &=\left|M(\|w_k\|^2)\kappa_{\alpha}\int_{\mathcal{C}}
y^{1-2\alpha} \nabla w_{k}\nabla v_{k} dxdy- M(\|v_k\|^2)\kappa_{\alpha}\int_{\mathcal{C}}
y^{1-2\alpha} |\nabla v_{k}|^2 dxdy\right|\\
&=\left|M(\|w_k\|^2)\kappa_{\alpha}\int_{I_{i_0}}\int_{\Omega}
y^{1-2\alpha} \nabla w_{k}\nabla v_{k} dxdy- M(\|v_k\|^2)\kappa_{\alpha}\int_{I_{i_0}}\int_{\Omega}
y^{1-2\alpha} |\nabla v_{k}|^2 dxdy\right|
\end{align*}
which implies that
\begin{align*}
|\langle \mc I_{M,\lambda}'(w_{k})-\mc I_{M,\lambda}'(v_{k}),v_{k}\rangle|&\leq \left|M(\|w_k\|^2)\kappa_{\alpha}\int_{I_{i_0}}\int_{\Omega}
y^{1-2\alpha} \nabla (w_{k}-v_k)\nabla v_{k} dxdy\right.\\
&\left. -(M(\|v_k\|^2)-M(\|w_k\|^2))\kappa_{\alpha}\int_{I_{i_0}}\int_{\Omega}
y^{1-2\alpha} |\nabla v_{k}|^2 dxdy\right|.
\end{align*}
Now using the fact that $M(t)=a+\varepsilon t$ in the last term of the above expression, we get
\begin{align*}
&|\langle \mc I_{M,\lambda}'(w_{k})-\mc I_{M,\lambda}'(v_{k}),v_{k}\rangle|\\
&=\left|M(\|w_k\|^2)\kappa_{\alpha}\int_{I_{i_0}}\int_{\Omega}
y^{1-2\alpha} \nabla (w_{k}-v_k)\nabla v_{k} dxdy- \epsilon(\|v_k\|^2-\|w_k\|^2)\kappa_{\alpha}\int_{I_{i_0}}\int_{\Omega}
y^{1-2\alpha} |\nabla v_{k}|^2 dxdy\right|.\\
\end{align*}
Now using $\|w_{k}\|\leq K$, Cauchy-Schwartz inequality and \eqref{tight} together with the compact inclusion
$H^{1}_{0, L}(I_{i_{0}}\times\Omega)$ into $L^{2}(I_{i_{0}}
\times\Omega)$, we have
\begin{eqnarray*}
&&\left|\langle \mc I_{M,\lambda}'(w_{k})-\mc I_{M,\lambda}'(v_{k}),v_{k}\rangle\right|\\&\leq &C_1\kappa_{\alpha}\left(\int_{I_{i_{0}}}{\int_{\Omega}{y^{1-2\alpha}
|\nabla(w_{k}-v_{k})|^{2}dxdy}}\right)^{\frac 12}\left(\int_{I_{i_{0}}}{\int_{\Omega}{y^{1-2\alpha}|\nabla v_{k}|^{2}dxdy}}\right)^{\frac12}+\epsilon \kappa_\alpha C_2,\\
&\leq & C_1 \kappa_{\alpha} \varepsilon+C_2 \kappa_\alpha \epsilon\leq C \kappa_\alpha\epsilon.
\end{eqnarray*}
On the other hand, using $|\langle \mc I_{M,\lambda}'(w_{k}),v_{k}\rangle|=o_k(1)$, we get
\[\left|\langle \mc I_{M,\lambda}'(v_{k}),v_{k}\rangle\right|\leq C\,\kappa_{\alpha}\,\varepsilon+o_k(1).\]
So, for $k$ sufficiently large,
\[\int_{\{y>r+i_{0}+1\}}{\int_{\Omega}{y^{1-2\alpha}|\nabla w_{k}|^{2}dxdy}}
\leq\int_{\mathcal{C}}{y^{1-2\alpha}|\nabla v_{k}|^{2}dxdy}=
\frac{\langle \mc I_{M,\lambda}'(v_{k}),v_{k}\rangle}{\kappa_{\alpha}}\leq C\,\varepsilon.\]
This is a contradiction with \eqref{hhh}, which proves the Lemma.
\end{proof}
\noi Now using Theorem \ref{lions} and Lemma \ref{tight1}, we prove the following proposition which shows the compactness of Palais-Smale sequence.
\begin{pro}\label{crcmp}
Suppose $\{w_k\}$ be a sequence in $H^1_{0,L}(\mathcal C)$ such that \\
\[
\mathcal I_{M,\lambda}(w_k)\rightarrow c \;\;\textrm{and}\;\; \mathcal I_{M,\lambda}^{\prime}(w_k)\rightarrow 0,\] where
\[
c<\left(\frac{1}{2}-\frac{1}{2^*_\alpha}\right)(a\kappa_\alpha S(\alpha, n))^\frac{2^*_\alpha}{2^*_\alpha-2}-\lambda^{\frac{2}{2-q}}\left(\frac{(4-q)\|f\|_{\frac{2^*_\alpha}{2^*_\alpha-q}}S(\alpha, n)^\frac{-q}{2}}{4q}\right)^\frac{2}{2-q} \left(\frac{2-q}{2}\right)\left(\frac{2q}{a}\right)^\frac{q}{2-q}
\]
is a positive constant, then $\{u_k\}$ possesses a strongly convergent subsequence.
\end{pro}
\begin{proof}
Let $\{w_k\}$ be a $(PS)_c$ sequence for $\mc I_{M,\lambda}$ in $H^1_{0,L}(\mathcal C)$ then it is easy to
 see that $\{w_k\}$ is bounded in $H^1_{0,L}(\mathcal C)$. Therefore there exists
 $w_0\in H^1_{0,L}(\mathcal C)$ such that $w_k\rightharpoonup w_0$ weakly in $H^1_{0,L}(\mathcal C)$, $ w_k(z,0)\rightarrow w_0(z,0)$
 in $L^\gamma(\Omega)$ for $\gamma\in [2, 2^*_\alpha)$ and $w_k\rightarrow w_0$ pointwise almost everywhere in $\Omega\times\{0\}$. Now from Theorem \ref{lions}, there exists two positive measures $\mu$ and $\nu$ on $\mathcal C$ such that
\begin{equation*}
y^{1-2\alpha}|\nabla w_k|^2  \rightarrow \mu\quad\mbox{and}\quad\left|w_k\right|^{2^*_\alpha}\rightarrow \nu,
\end{equation*}
Moreover, we have a countable index set $J$, positive constants $\{\nu_j\}_{j\in J}$ and $\{\mu_j\}_{j\in J}$ such that
\begin{equation*}
\nu=\left|w_0\right|^{2^*_\alpha}+\sum_{j\in J} \nu_j\delta_{z_j}, \;\; \; \text{and}\;\mu\geq y^{1-2\alpha}|\nabla w_0|^2 +\sum_{j\in J} \mu_j\delta_{z_j},\qquad \mu_j\geq S(\alpha, n)\nu^{\frac{2}{2^*_\alpha}}_{j}.
\end{equation*}
Our goal is to show that $J$ is empty. Suppose not then for any $j\in J$ we can consider the cut-off functions, $\psi_{\epsilon, j}(z)$, centered at $z_j$ such that $0\leq \psi_{\epsilon, j}(z)\leq 1$, $\psi_{\epsilon, j}(z)=1$ in $B_{\frac{\epsilon}{2}}(z_j)$, $\psi_{\epsilon, j}(z)=0$ in $B^c_{\epsilon}(z_j)$,  and $|\nabla \psi_{\epsilon, j}(z)|\leq {C}/{\epsilon}$. Then we have
\begin{equation}\label{flim}
\displaystyle \lim_{\epsilon\rightarrow 0}\lim_{k\rightarrow\infty}\int_{\Omega\times\{0\}}f(z)|w_k(z,0)|^{q-1}\psi_{\epsilon,j}w_k(z,0)dz=0.
\end{equation}
Now using \eqref{flim} and boundedness of the sequence $\{w_k\}$,
\begin{align*}
0&=\displaystyle\lim_{\epsilon\rightarrow 0}\displaystyle \lim_{k\rightarrow \infty} \langle \mc I^\prime_{M,\lambda}(w_k), \psi_{\epsilon,j}w_k\rangle\\
&=\displaystyle\lim_{\epsilon \rightarrow 0}\displaystyle \lim_{k\rightarrow \infty}\left\{(a+\epsilon \|w_k\|^2)\kappa_\alpha\int_{\mathcal C}y^{1-2\alpha}\nabla w_k \nabla (\psi_{\epsilon, j}w_k)dzdy-\int_{\Omega\times\{0\}}|w_k(z,0)|^{2^*_\alpha}\psi_{\epsilon,j}(z,0)dz\right\}
\end{align*}
which implies
\begin{align*}
0&=\displaystyle\lim_{\epsilon \rightarrow 0}\displaystyle \lim_{k\rightarrow \infty}\left\{a \kappa_\alpha\int_{\mathcal C}y^{1-2\alpha}|\nabla w_k|^2 \psi_{\epsilon, j}dzdy+a \kappa_\alpha\int_{\mathcal C}y^{1-2\alpha}w_k\nabla w_k \nabla \psi_{\epsilon, j}dzdy\right.\\&+\epsilon \|w_k\|^2 \kappa_\alpha\int_{\mathcal C}y^{1-2\alpha}|\nabla w_k |^2 \psi_{\epsilon, j}dzdy+
\epsilon \|w_k\|^2 \kappa_\alpha\int_{\mathcal C}y^{1-2\alpha}w_k \nabla w_k \nabla \psi_{\epsilon, j}dzdy
\\ &-\left.\int_{\Omega\times\{0\}}|w_k(z,0)|^{2^*_\alpha}\psi_{\epsilon,j}(z,0)dz\right\}.\\
\end{align*}
Now using
\begin{eqnarray*}
0&\leq&\displaystyle \lim_{k\to\infty} \left|\int_{\mathcal{C}_{\Omega}}{y^{1-2\alpha}w_k\nabla w_{k}\nabla \psi_{\epsilon, j} dzdy} \right|\\
&\leq &\displaystyle \lim_{k\to
\infty}\left(\int_{\mathcal{C}_{\Omega}}y^{1-2\alpha}|\nabla w_{k}|^{2}dzdy\right)^{1/2}
\left(\int_{B_{\epsilon}(z_j)}{y^{1-2\alpha}|\nabla
\psi_{\epsilon, j}|^{2}|w_{k}(z,0)|^{2}dzdy}\right)^{1/2}\longrightarrow0,
\end{eqnarray*}
 we get
 \begin{align*}
0&\geq\displaystyle\lim_{\epsilon \rightarrow 0}\displaystyle \lim_{k\rightarrow \infty}\left\{a \kappa_\alpha\int_{\mathcal C} \psi_{\epsilon, j}d\mu+a \kappa_\alpha\int_{\mathcal C}y^{1-2\alpha}w_k\nabla w_k \nabla \psi_{\epsilon, j}dzdy\right.\\&+\epsilon \|w_k\|^2 \kappa_\alpha\int_{\mathcal C}\psi_{\epsilon, j}d\mu+
\epsilon \|w_k\|^2\kappa_\alpha\int_{\mathcal C}y^{1-2\alpha}w_k \nabla w_k \nabla \psi_{\epsilon, j}dzdy
-\left.\int_{\Omega\times\{0\}}\psi_{\epsilon,j}d\nu\right\}\\
&
\geq a\kappa_\alpha\mu_j-\nu_j.
\end{align*}
From the relation $\mu_j\geq S(\alpha, n)\nu^{\frac{2}{2^*_\alpha}}_{j}$ implies $\mu_j\geq \left(a^2 \kappa^2_\alpha S(\alpha, n)^{2^*_\alpha}\right)^\frac{1}{2^*_\alpha-2}$ or $\mu_j=0$. We claim that
$$\mu_j\geq \left(a^2 \kappa^2_\alpha S(\alpha, n)^{2^*_\alpha}\right)^\frac{1}{2^*_\alpha-2}$$
 is not possible to hold. We prove by contradiction. Suppose
 \begin{align}\label{lereq}
 \mu_j\geq \left(a^2 \kappa^2_\alpha S(\alpha, n)^{2^*_\alpha}\right)^\frac{1}{2^*_\alpha-2}.
\end{align}
Then
\begin{align*}
c&=\displaystyle\lim_{k\rightarrow \infty} \left\{\mc I_{M,\lambda}(w_k)-\frac{1}{4}\langle \mc I_{M,\lambda}^{\prime}(w_k), w_k\rangle\right\}\\
%&\geq \displaystyle \lim_{k\rightarrow \infty} \left\{ \frac{1}{4}a\|w_k\|^2-\lambda \left(\frac{1}{q}-\frac{1}{4}\right)\int_{\Omega \times \{0\}}f(z)|w_k(z,0)|^qdz \right\}\\
&\geq \frac{1}{4}a\left(\|w_0\|^2+\kappa_\alpha\sum_{j\in J}\mu_j\delta_{z_j}\right)+ \left(\frac{1}{4}-\frac{1}{2^*_\alpha}\right)\left(\|w_0\|_{2^*_\alpha}^{2^*_\alpha}+\sum_{j\in J}\nu_j\delta_{z_j}\right)\\&-\lambda \left(\frac{1}{q}-\frac{1}{4}\right)\int_{\Omega \times \{0\}}f(z)|w_0(z,0)|^qdz.\\
\end{align*}
Using \eqref{lereq}, we have
\begin{align*}
c&\geq \frac14 a\kappa_\alpha\mu_{j_0}+\left(\frac{1}{4}-\frac{1}{2^*_\alpha}\right)\nu_{j_0}+\frac{1}{4}a\|w_0\|^2-\lambda \left(\frac{1}{q}-\frac{1}{4}\right)\|f\|_{L^{\frac{2^*_\alpha}{2^*_\alpha-q}}}(\sqrt{\kappa_\alpha}S(\alpha, n))^{-q}\|w_0\|^q
\\&\geq \left(\frac{1}{2}-\frac{1}{2^*_\alpha}\right)(a\kappa_\alpha S(\alpha, n))^\frac{2^*_\alpha}{2^*_\alpha-2}-\lambda^\frac{2}{2-q}\left(\frac{(4-q)\|f\|_{\frac{2^*_\alpha}{2^*_\alpha-q}}
(\sqrt{\kappa_\alpha}S(\alpha, n))^{-q}}{4q}\right)^\frac{2}{2-q} \left(\frac{2-q}{2}\right)\left(\frac{2q}{a}\right)^\frac{q}{2-q},
\end{align*}
which is a contradiction. Hence $J$
is empty and $$\ds \int_{\Omega\times\{0\}}|w_k(z,0)|^{2^*_\alpha} dz\rightarrow \ds \int_{\Omega\times\{0\}}|w_0(z,0)|^{2^*_\alpha} dz.$$

\end{proof}

\noindent \textbf{Proof of Theorem \ref{22mht1} (i)} Assume $\lambda_0=\min\{\lambda_1, \lambda_2, \lambda_3\}$. Now as the functional is bounded below in $\mathcal N_\lambda$, we minimize the functional $\mc I_{M, \lambda}$
 in $\mathcal N_\lambda$ and using Proposition \ref{prp1}, Lemma \ref{3a} and Proposition \ref{crcmp}, we get the minimizer $w_0$ of  $\mc I_{M, \lambda}$ in $\mc N_{\lambda}$ for $\lambda\in(0, \lambda_0)$ with $\mc I_{M, \lambda}(w_0)<0$. Now we claim that $w_0 \in \mathcal N_\lambda^+$ for $\lambda\in (0,\lambda_0)$.
 If not then $w_0\in \mc N_\lambda^-$. Note that using $w_0\in \mc N_\lambda$ and $\mc I_{M, \lambda}(w_0)<0$ we get $w_0 \in H^+$. Therefore from Lemma \ref{L37}, we get $t^-(w_0)>t^+(w_0)>0$ such that $t^-w_0\in \mc N_\lambda^-$ and $t^+w_0\in \mc N_\lambda^+$ which implies $t^-=1$ and $t^+<1$. Therefore we can find $t_0\in (t^+, t^-)$ such that
\[
 \mathcal I_{M,\lambda}(t^+w_0)=\displaystyle\min_{0\leq t\leq t^-} \mathcal I_{M,\lambda}(tw_0)<\mathcal I_{M,\lambda}(t_0w_0)\leq \mathcal I_{M,\lambda}(t^-w_0)=\mathcal I_{M,\lambda}(w_0)=\theta_\lambda^+
 \]
 which is a contradiction. Hence $w_0\in \mc N_\lambda^+.$ Since $\mc I_{M, \lambda}(w)=\mc I_{M, \lambda}(|w|)$, we can assume that $w_0\geq0$. Now using the fact that $M(t)>a$ and strong maximum principle (see \cite{MR2270163}),  we get $w_0>0$.
 \noindent Now the following Lemma shows that $w_0$ is indeed a local minimizer of $\mc I_{M,\lambda}$ in $H^1_{0,L}(\mathcal C)$.
 \begin{lem}
The function $w_0 \in \mathcal N_\lambda^+$ is a local minimum of $\mathcal I_{M,\lambda}(w)$ in $H^1_{0, L}(\mathcal C)$ for $\lambda<\lambda_0$.
\end{lem}
\begin{proof}
Since $w_0 \in \mathcal N_\lambda^{+}$, we have $t^+(w_0)=1
<t_*(w_0)$. Hence by continuity of $w\mapsto t_*(w)$, given
$\epsilon>0$, there exists $\delta=\delta(\epsilon)>0$ such that $1+\epsilon< t_*(w_0-w)$
for all $\|w\|<\delta$. Also, from Lemma \ref{zii}, for $\delta>0$
small enough, we obtain a $C^1$ map $t: \mc {B}(0,\delta)\rightarrow \mathbb R^+$
such that $t(w)(w_0-w)\in  \mathcal N_\lambda$, $t(0)=1$. Therefore, for
$\delta>0$ small enough we have $t^+(w_0-w)=
t(w)<1+\epsilon<t_*(w_0-w)$ for all $\|w\|<\delta$. Since $t_*(w_0-w)>1$,
we obtain $\mathcal I_{M,\lambda}(w_0)\leq \mathcal I_{M,\lambda}(t_{1}(w_0-w)(w_0-w))\leq \mathcal I_{M,\lambda}(w_0-w)$
for all $\|w\|<\delta$. This shows that $w_0$ is a local minimizer for
$\mathcal I_{M,\lambda}$ in $H^1_{0, L}(\mathcal C)$.
\end{proof}
\section{Existence of second solution in \texorpdfstring{$\mathcal N_\lambda^-$}{sclg}}
\noi Now we show the existence of second solution in $\mathcal N_\lambda^-$. The following Lemma gives the critical level to show the second solution
by considering the mountain pass structure around first solution.\\
\noindent
Let $\sum=\{z\in \Omega\;|\;f(z)>0\}$ be an open set with positive measure. Consider the test functions as
$\eta \in C_c^{\infty}(\mathcal C_{\sum})$, where $\mathcal C_{\sum}=\sum\times (0, \infty)$ such that $0\leq \eta(z,y)\leq 1$
in $\mathcal C_{\sum}$ and $(supp f^+\times \{y>0\})\cap\{(z,y)\in \mathcal C_{\sum}: \eta=1\}\neq \emptyset$.
Moreover, for $\rho>0$ small, $\eta(z,y)=1$ on $\mc B_{\rho}(0)$ and $\eta(z,y)=0$ on $\mc B^c_{2\rho}(0)$. We take $\rho$ small enough such that $\mc B_{2\rho}(0)\subset \mathcal C_{\sum}$. Consider $w_{\epsilon,\eta}=\eta w_\epsilon\in H^1_{0,L}(\mathcal C)$, where $w_\epsilon$ is defined as in \eqref{extrml}. Then for $\lambda\in(0, \lambda_0)$, we have the following lemma.
\begin{lem}\label{II}
Let $w_{0}$ be the local minimum for the functional $\mc I_{M,\lambda}$ in $H^1_{0,L}(\mc C)$. Then for every $r>0$ and a.e. $\eta \in \sum$\;
there exists $\epsilon_{0} = \epsilon_{0}(r, \eta) > 0$ s.t.
\begin{equation*}
\mathcal I_{M,\lambda}(w_{0}+r\;w_{\epsilon,\eta}) <c_0, \;\;\textrm{for}\;\;\epsilon \in (0,\epsilon_{0}),
\end{equation*}
\textrm{where}
\[
c_0= \left(\frac{1}{2}-\frac{1}{2^*_\alpha}\right)(a\kappa_\alpha S(\alpha, n))^\frac{2^*_\alpha}{2^*_\alpha-2}-\lambda^{\frac{2}{2-q}}\left(\frac{(4-q)\|f\|_{\frac{2^*_\alpha}{2^*_\alpha-q}}
(\sqrt{\kappa_\alpha}S(\alpha, n))^{-q}}{4q}\right)^\frac{2}{2-q} \left(\frac{2-q}{2}\right)\left(\frac{2q}{a}\right)^\frac{q}{2-q}.
\]
\end{lem}
\begin{proof}
 From \eqref{fel},
\begin{align*}
\mathcal I_{M,\lambda}(w_{0} + r\;w_{\epsilon,\eta})&=\frac{a}{2}\| w_{0} + r\;w_{\epsilon,\eta}\|^{2} + \frac{\epsilon}{4}\| w_{0} + r\;w_{\epsilon,\eta}\|^{4}- \frac{\lambda}{q}\int_{\Omega\times\{0\}} f(z)|w_{0}+ r\;w_{\epsilon, \eta}|^{q}dz \\
&- \frac{1}{2^*_\alpha}\int_{\Omega\times\{0\}} |w_{0} + r\;w_{\epsilon,\eta}|^{{2^*_\alpha}}dz\\
&= \frac{a}{2}\| w_{0}\|^{2} + \frac{a}{2}r^{2}\|w_{\epsilon, \eta}\|^{2} + a\;r \langle w_{0}, w_{\epsilon,\eta}\rangle + \frac{\epsilon}{4}\| w_{0}\|^{4}
 + \frac{\epsilon}{4} r^{4}\|w_{\epsilon, \eta}\|^{4}\\
 & + \;\epsilon r^{2} \langle w_{0}, w_{\epsilon,\eta}\rangle ^{2}
 + \frac{\epsilon}{2} r^{2}\|w_{0}\|^{2}\|w_{\epsilon, \eta}\|^{2} + \epsilon r^{3} \|w_{\epsilon, \eta}\|^{2} \langle w_{0}, w_{\epsilon, \eta}\rangle +\; \epsilon r \|w_{0}\|^{2}
 \langle w_{0}, w_{\epsilon, \eta}\rangle\\ &- \frac{\lambda}{q} \int_{\Omega\times\{0\}} f(z) |(w_{0} +  r w_{\epsilon, \eta})(z,0)|^{q}dz - \frac{1}{{2^*_\alpha}} \int_{\Omega\times\{0\}} |(w_{0}+rw_{\epsilon, \eta})(z,0)|^{{2^*_\alpha}}dz.
 \end{align*}
 Using the fact that $w_{0}$ is a solution of problem \eqref{E}, we get
 \begin{align*}
 \mathcal I_{M,\lambda}(w_{0} + r\;w_{\epsilon,\eta})&\leq  \mc I_{M,\lambda}(w_0) + \frac{a}{2}r^{2}\|w_{\epsilon, \eta}\|^2 + \frac{\epsilon}{4}r^{4}\|w_{\epsilon, \eta}\|^{4} + \epsilon\;r^{2}\| w_{0}\|^{2}\|w_{\epsilon, \eta}\|^2 +\frac{\epsilon}{2}r^{2}\| w_{0}\|^{2}\|w_{\epsilon, \eta}\|^2\\
 &+ \epsilon\;r^{3} \|w_{\epsilon, \eta}\|^{3}\| w_{0}\| -\frac{\lambda}{q} \left(\int_{\Omega\times\{0\}} f(z)( |w_{0} + r\;w_{\epsilon,\eta}|^{q}-|w_0|^q-qr|w_{0}|^{q-1} w_{\epsilon,\eta} )(z,0)dz\right)
 \\&- \frac{1}{{2^*_\alpha}}\left(\int_{\Omega\times\{0\}} (|w_0+rw_{\epsilon,\eta}|^{2^*_\alpha}-|w_0|^{2^*_\alpha}-{2^*_\alpha}rw_0^{2^*_\alpha-1} w_{\epsilon,\eta})(z,0) dz\right).
 \end{align*}
 Let $f(x)>0$ on $\sum$ and  $\|w_0\|=R$ together with Young's inequality, we get
\begin{align*}
 \mathcal I_{M,\lambda}(w_{0} + r\;w_{\epsilon,\eta})&\leq  \mc I_{M,\lambda}(w_0)+\frac{a}{2}r^{2}\|w_{\epsilon, \eta}\|^2+\frac{\epsilon}{4}r^{4}\|w_{\epsilon, \eta}\|^{4}+\frac{3\epsilon}{2}r^{2}R^{2}\|w_{\epsilon, \eta}\|^2+ \epsilon R\;r^{3} \|w_{\epsilon, \eta}\|^{3}\\&-\frac{1}{{2^*_\alpha}}r^{2^*_\alpha}\int_{\Omega\times\{0\}}|w_{\epsilon,\eta}(z,0)|^{2^*_\alpha}dz-C r^{2^*_\alpha-1}\int_{\Omega\times\{0\}} |w_{\epsilon,\eta}(z,0)|^{2^*_\alpha-1}dz.\\
 &\leq\mc I_{M,\lambda}(w_0)+\frac{a}{2}r^{2}\|w_{\epsilon, \eta}\|^2+\frac{5\epsilon}{4}r^{4}\|w_{\epsilon, \eta}\|^{4}-\frac{1}{{2^*_\alpha}}r^{2^*_\alpha}\int_{\Omega\times\{0\}}|w_{\epsilon,\eta}(z,0)|^{2^*_\alpha}dz\\
 &-C r^{2^*_\alpha-1}\int_{\Omega\times\{0\}} |w_{\epsilon,\eta}(z,0)|^{2^*_\alpha-1}dz+\frac{5\epsilon}{2}R^4.
 \end{align*}
 Now assume
 \begin{align*}
 g(t)=\frac{a}{2}t^{2}\|w_{\epsilon, \eta}\|^2+\frac{5\epsilon}{4}t^{4}\|w_{\epsilon, \eta}\|^{4}-
 C t^{2^*_\alpha-1}\ds \int_{\Omega\times\{0\}} |w_{\epsilon,\eta}(z,0)|^{2^*_\alpha-1}dz-\frac{1}{{2^*_\alpha}}t^{2^*_\alpha}\ds \int_{\Omega\times\{0\}}|w_{\epsilon,\eta}(z,0)|^{2^*_\alpha}dz.
 \end{align*}
 Next we claim the following\\
 \textbf{Claim:} There exists $t_\epsilon>0$ and $t_1, t_2>0$ (independent of $\epsilon, \lambda$) such that
 \begin{align*}
 g(t_\epsilon w_{\epsilon,\eta})=\displaystyle \sup_{t\geq 0}g(tw_{\epsilon,\eta})\;\textrm{and}\; \frac{d}{dt}g(tw_{\epsilon,\eta})\mid_{t=t_\epsilon}=0.
 \end{align*}
 and  $0<t_1\leq t_\epsilon\leq t_2<\infty$.\\
 \noi  Since $\displaystyle \lim_{t\rightarrow \infty} g(t)=-\infty$ and $\displaystyle \lim_{t\rightarrow 0^+} g(t)>0$.
 Therefore there exists $t_\epsilon>0$ such that
 \begin{align}\label{mm}
 g(t_\epsilon w_{\epsilon,\eta})=\displaystyle \sup_{t\geq 0}g(tw_{\epsilon,\eta})\;\textrm{and}\; \frac{d}{dt}g(tw_{\epsilon,\eta})\mid_{t=t_\epsilon}=0.
 \end{align}
 From \eqref{mm}, we get the following
 \begin{align}\label{der0}
 &{a}t\|w_{\epsilon, \eta}\|^2+{5\epsilon}t^{3}\|w_{\epsilon, \eta}\|^{4}=
 C t^{2^*_\alpha-2}\ds \int_{\Omega\times\{0\}} |w_{\epsilon,\eta}(z,0)|^{2^*_\alpha-1}dz+t^{2^*_\alpha-1}\ds \int_{\Omega\times\{0\}}|w_{\epsilon,\eta}(z,0)|^{2^*_\alpha}dz,
 \\ \label{ddn}
 &{a}\|w_{\epsilon, \eta}\|^2<
 C t^{2^*_\alpha-3}\ds \int_{\Omega\times\{0\}} |w_{\epsilon,\eta}(z,0)|^{2^*_\alpha-1}dz+C_1t^{2^*_\alpha-2}\ds \int_{\Omega\times\{0\}}|w_{\epsilon,\eta}(z,0)|^{2^*_\alpha}dz-{C_3\epsilon}t^{2}\|w_{\epsilon, \eta}\|^{4}.
\end{align}
 From \eqref{ddn}, it is clear that $t_\epsilon$ is bounded below that is there exist constants $t_1>0$, independent of $\varepsilon, \lambda$ such that $0<t_1\leq t_\epsilon$. Also from \eqref{der0}, we have
 \begin{align*}
 \frac{a}{t^2}\|w_{\epsilon, \eta}\|^2+{5\epsilon}\|w_{\epsilon, \eta}\|^{4}=
 C t^{2^*_\alpha-5}\ds \int_{\Omega\times\{0\}} |w_{\epsilon,\eta}(z,0)|^{2^*_\alpha-4}dz+t^{2^*_\alpha-1}\ds \int_{\Omega\times\{0\}}|w_{\epsilon,\eta}(z,0)|^{2^*_\alpha}dz.
 \end{align*}
 Since $2^*_\alpha>4$, there exists $t_2>0$, independent of $\varepsilon, \lambda$ such that $t_\epsilon\leq t_2<\infty$.\qed\\
 \noi Now from \cite{MR2911424, MR3117361} we get that the family $\{\eta w_\epsilon\}$ and its trace on $\{y=0\}$ namely $\{\eta u_\epsilon\}$ satisfy
 $$
 \|\eta w_\epsilon\|^2=\|w_\epsilon\|^2+O(\epsilon^{n-2\alpha}),\;\;
 \|\eta u_\epsilon\|^{2^*_\alpha-1}_{L^{2^*_\alpha-1}(\Omega)}\geq C\epsilon^\frac{n-2\alpha}{2}$$
  and
  $$\displaystyle \int_{\Omega\times\{0\}}|\eta w_\epsilon(z,0)|^{2^*_\alpha}dz=\ds \int_{\mathbb \R^n}\left(\frac{\epsilon}{\epsilon^2+|z|^2}\right)^n dz+O(\epsilon  ^n).
  $$
Now using these estimates we get,
 \begin{align*}
 \displaystyle \sup_{t\geq 0}g(tw_{\epsilon,\eta})&=g(t_\epsilon w_{\epsilon,\eta})\leq \displaystyle\left(\frac{a}{2}t_\epsilon^{2}\|w_{\epsilon}\|^2
 -\frac{1}{2^*_\alpha}t_\epsilon^{2^*_\alpha}\int_{\Omega\times\{0\}}|w_{\epsilon}(z,0)|^{2^*_\alpha}dz\right)
 +C_2{\epsilon}\\&-C\epsilon^\frac{n-2\alpha}{2}+O(\epsilon^{n-2\alpha})+O(\epsilon^n)\\
 &\leq \left(\frac{1}{2}-\frac{1}{2^*_\alpha}\right)(a\kappa_\alpha S(\alpha, n))^\frac{2^*_\alpha}{2^*_\alpha-2}+C_2\epsilon-C\epsilon^\frac{n-2\alpha}{2},
  \end{align*}
where $C_2, C>0$ are positive constants independent of $\epsilon, \lambda$. Now choose $\varepsilon=\lambda^\frac{2}{2-q}$. Then
\begin{align*}
C_2\epsilon-C\epsilon^\frac{n-2\alpha}{2}&=C_2\lambda^\frac{2}{2-q}-C\lambda^\frac{n-2\alpha}{2-q}
=\lambda^\frac{2}{2-q}\left( C_2-C\lambda ^\frac{n-2\alpha-2}{2-q}\right).
\end{align*}
Since $2\alpha<n<4\alpha$, implies $n-2\alpha-2<0$. Therefore choosing $0<\lambda<\lambda_4$ such that
$$
\left(C_2+\left(\frac{(4-q)\|f\|_{\frac{2^*_\alpha}{2^*_\alpha-q}}(\sqrt{\kappa_\alpha}S(\alpha, n))^{-q}}{4q}\right)^\frac{2}{2-q} \left(\frac{2-q}{2}\right)\left(\frac{2q}{a}\right)^\frac{q}{2-q}\right)\lambda_4^\frac{2+2\alpha-n}{2-q}<C
$$
and using $\mc I_{M,\lambda}(w_0)<0$, we get $\mathcal I_{M,\lambda}(w_{0} + r\;w_{\epsilon,\eta}) \leq c_0$. This proves the Proposition.
\end{proof}
\noindent \textbf{Proof of Theorem \ref{22mht1} (ii):} Consider the following
\begin{eqnarray*}
  W_{1} &=& \left\{w \in H_{0, L}^{1}(\mathcal C)\setminus\{0\} \big{|} \frac{1}{\|w\|}t^{-}\left(\frac{w}{\|w\|}\right) > 1\right\} \cup \{0\}, \\
   W_{2} &=& \left\{w \in H_{0,L}^{1}(\mathcal C)\setminus\{0\} \big{|} \frac{1}{\|w\|}t^{-}\left(\frac{w}{\|w\|}\right) < 1\right\}.
\end{eqnarray*}
Then $\mc N_{\lambda}^{-}$ disconnects $ H_{0,L}^{1}(\mathcal C)$ in two connected components $W_{1}$ and $W_{2}$
and $ H_{0,L}^{1}(\mathcal C)\setminus \mc N_{\lambda}^{-} = W_{1} \cup W_{2}.$ For each $w \in \mc N_{\lambda}^{+},$ we have $1< t_{\max}(w) < t^{-}(w).$
 Since $t^{-}(w) = \frac{1}{\|w\|}t^{-}\left(\frac{w}{\|w\|}\right),$ then $\mc N_{\lambda}^{+} \subset W_{1}.$
In particular, $w_0 \in W_{1}.$  Now we claim that there exists $l_0>0$ such that $w_{0} + l_{0}w_{\epsilon, \eta} \in W_{2}. $

 \noi First, we find a constant $c>0$ such that $0 < t^{-} \left(\frac{w_{0}+l\;w_{\epsilon,\eta}}{\|w_{0}+l\;w_{\epsilon,\eta}\|}\right)<c.$ Otherwise, there exists a sequence $\{l_{n}\}$ such that $l_{n} \rightarrow \infty$ and $t^{-}\left(\frac{w_{0}+l_{n}\;w_{\epsilon,\eta}}{\|w_{0}+l_{n}\;w_{\epsilon,\eta}\|}\right) \rightarrow \infty$ as $n \rightarrow\infty.$ Let $v_{n} = \frac{w_{0}+l_{n}\;w_{\epsilon,\eta}}{\|w_{0}+l_{n}\;w_{\epsilon,\eta}\|}.$ Since $t^{-}(v_{n})v_{n} \in \mc N_{\lambda}^{-} \subset \mc N_{\lambda}$ and by the Lebesgue dominated convergence theorem,
\begin{eqnarray*}
  \int_{\Omega\times\{0\}}v_{n}(z,0)^{2_\alpha^*} &=& \frac{1}{\|w_{0}+l_{n}\;w_{\epsilon,\eta}\|^{2_\alpha^*}}\int_{\Omega\times\{0\}} (w_{0}+l_{n}\;w_{\epsilon,\eta})(z,0)^{2_\alpha^*}dz \\
   &=& \frac{1}{\|w_{0}+l_{n}\;w_{\epsilon,\eta}\|^{2_\alpha^*}}\int_{\Omega\times\{0\}} \left(\frac{w_{0}}{l_{n}}+w_{\epsilon,\eta}\right)(z,0)^{2_\alpha^*} dz \\
   &\rightarrow& \frac{\int_{\Omega\times\{0\}}(w_{\epsilon, \eta}(z,0))^{2_\alpha^*}dz}{\|w_{\epsilon, \eta}\|^{2_\alpha^*}}\;\; \textrm{as}\;\; n \rightarrow \infty.
\end{eqnarray*}
Now
\begin{align*}
  \mathcal{I}_{M, \lambda}(t^{-}(v_{n})v_{n}) &= \frac{1}{2}a(t^{-}(v_n))^2\|v_n\|^2+\frac{1}{4}\epsilon(t^{-}(v_n))^4\|v_n\|^4-\frac{(t^{-}(v_{n}))^{q}}{q}\lambda \int_{\Omega\times\{0\}}f(z) v_{n}(z,0)^q\;dz\\& - \frac{(t^{-}(v_{n}))^{2^*_\alpha}}{2^*_\alpha}\int_{\Omega\times\{0\}} v_{n}(z,0)^{2^*_\alpha}dx \rightarrow - \infty\;\; \textrm{as}\;\; n \rightarrow \infty,
\end{align*}
this contradicts that $\mc I_{M,\lambda}$ is bounded below on $\mc N_{\lambda}.$ Let
\begin{equation*}
    l_{0} = \frac{|c^{2}-\|w_{0}\|^{2}|^{\frac{1}{2}}}{\|w_{\epsilon,\eta}\|} + 1,
\end{equation*}
then
\begin{align*}
  \|w_{0}+l_{0}w_{\epsilon,\eta}\|^{2} &= \|w_{0}\|^{2} + (l_{0})^{2}\|w_{\epsilon,\eta}\|^{2}+2l_{0}\langle w_{0}, w_{\epsilon,\eta}\rangle >\|w_{0}\|^{2} + |c^{2}-\|w_{0}\|^{2}| + 2l_{0}\langle w_{0}, w_{\epsilon,\eta}\rangle\\
   &>c^{2}>\left(t^{-}\left(\frac{w_{0}+l_{0}w_{\epsilon,\eta}}{\|w_{0}+l_{0}w_{\epsilon,\eta}\|}\right)\right)^{2}
\end{align*}
that is $w_{0}+l_{0}w_{\epsilon,\eta} \in W_{2}.$\\
 \noindent Now, we define $\beta = \displaystyle \inf_{\gamma \in \Gamma} \displaystyle \max_{s \in [0,1]} \mc I_{M,\lambda}(\gamma(s)),$
where $\Gamma = \{\gamma \in C([0,1]; H_{0, L}^{1}(\mathcal C)) | \gamma(0) = w_0 \;\textrm{and}\; \gamma(1) = w_0+l_{0}w_{\epsilon,\eta}\}$ {and $\lambda_{00}=\ds \min\{\lambda_0, \lambda_4\}$.}  Define a path $\gamma_{0} = w_0+t\;l_{0}w_{\epsilon,\eta}$ for $t \in [0,1],$ then $\gamma_{0} \in \Gamma$ and there exists $t_{0} \in (0,1)$
  such that $\gamma_{0}(t_{0}) \in \mc N_{\lambda}^{-},$ we have $\beta \geq \theta_{\lambda}^{-}.$ Moreover, by Lemma \ref{II},
    $\theta_{\lambda}^{-} \leq \beta < c_0$ {for $0<\lambda<\lambda_{00}$,} where $c_0$ is defined in Lemma \ref{II}.
Now similar to the Proposition \ref{prp1}, one can show the existence of Palais-Smale sequence
$\{w_{k}\} \subset \mc N_{\lambda}^{-}$. Since
\[
\theta_{\lambda}^{-} < \left(\frac{1}{2}-\frac{1}{2^*_\alpha}\right)(a\kappa_\alpha S(\alpha, n))^\frac{2^*_\alpha}{2^*_\alpha-2}-\lambda^{\frac{2}{2-q}}\left(\frac{(4-q)\|f\|_{\frac{2^*_\alpha}{2^*_\alpha-q}}
(\sqrt{\kappa_\alpha}S(\alpha, n))^{-q}}{4q}\right)^\frac{2}{2-q} \left(\frac{2-q}{2}\right)\left(\frac{2q}{a}\right)^\frac{q}{2-q},
\]
by Proposition \ref{crcmp}, there exist a subsequence ${w_k}$ and $w_1$ in  $H_{0, L}^{1}(\mathcal C)$ such that $w_{k} \rightarrow w_1$ strongly in $ H_{0, L}^{1}(\mathcal C).$ Now using Corollary \ref{nlclosed}, $w_1 \in \mc N_{\lambda}^{-}$ and $\mc  I_{M,\lambda}(w_{k}) \rightarrow \mc I_{M,\lambda}(w_1) = \theta_{\lambda}^{-}\; \textrm{as}\; k \rightarrow \infty.$ Therefore $w_1$ is also a solution. Moreover, $\mc I_{M,\lambda}(w)=\mc I_{M,\lambda}(|w|)$, we may assume that $w_1\geq 0$. Again using strong maximum principle (see \cite{MR2270163}) $w_1$ is positive solution
 of the problem $\eqref{E}$. Since $\mc N_\lambda ^+\cap \mc N_\lambda^-=\emptyset$, $w_0$ and $w_1$ are distinct. This proves Theorem \ref{22mht1}.

\end{document}